\documentclass[reqno,12pt]{amsart} 
\usepackage[margin=2.85cm,footskip=1cm]{geometry}

\usepackage{enumerate}

\usepackage{amsmath} 

\usepackage{amsfonts}   

\usepackage{amssymb}      

\usepackage{amscd}      

\usepackage{amsthm}      
   
\usepackage{epsfig}      

\usepackage{amstext}      

\usepackage[all]{xy} 
\CompileMatrices 

\usepackage[latin1]{inputenc}

\usepackage{multirow}

 \newcommand{\Ext}{\operatorname{Ext}}

\newcommand{\N}{\mathbb{N}}

\newcommand{\Z}{\mathbb{Z}}
\newcommand{\C}{\mathbb{C}}

   \theoremstyle{plain}
   \newtheorem{theorem}{Theorem}[section]
   \newtheorem{prop}[theorem]{Proposition}
   \newtheorem{lemma}[theorem]{Lemma}  
   \newtheorem{cor}[theorem]{Corollary}
   \theoremstyle{definition}
   
   \newtheorem{defn}[theorem]{Definition}
   \newtheorem{example}[theorem]{Example}
   \theoremstyle{remark}
   
   \newtheorem{remark}[theorem]{Remark}

   \numberwithin{equation}{section}

\title[On reduced amalgamated free products...]{On reduced amalgamated free products of $ C^* $-algebras and the MF property}
  
\author{Jonas Andersen Seebach}

        \date{\today}

\date{\today}

\email{curly@imf.au.dk}
\address{Institut for matematiske fag, Ny Munkegade, 8000 Aarhus C, Denmark}

\begin{document}

\maketitle

\begin{abstract} We establish the MF property of the reduced group $ C^* $-algebra of an amalgamated free product of countable Abelian discrete groups. This result is then used to give a characterization of the amalgamated free products of Abelian groups for which the BDF semigroup of the reduced group $ C^* $-algebra is a group. Along the way we get a tensor product factorization of the corresponding group von Neumann algebra. We end the exposition by applying the ideas from the first part to give a few more examples of groups with a reduced group $ C^* $-algebra which is MF.
\end{abstract}

\section{Introduction}
Since Anderson \cite{A} found the first example of a $ C^* $-algebra with non-invertible extensions by the compact operators $ \mathbb{K} $ on a separable Hilbert space, several new examples of $ C^* $-algebras with this property have been discovered. Here invertibility is in the sense of Brown, Douglas and Fillmore, see, e.g., \cite{Ar}. Most notably Haagerup and Thorbj{\o}rnsen \cite{HT} showed that there is a non-invertible extension of $ C^*_r(\mathbb{F}_n) $ by $ \mathbb{K} $ where $ C^*_r(\mathbb{F}_n) $ is the reduced group $ C^* $-algebra of the free group $ \mathbb{F}_n $ on $ n \in \{ 2, 3,  \cdots \} \cup \{\infty\} $ generators, thus providing the first 'non-artificial' example of such an algebra.

This paper grew out of an investigation of the extensions of the reduced group $ C^* $-algebras of the so-called torus knot groups which are the one-relator groups with presentations $ \langle a_1, a_2 \mid a_1^k a_2 ^{-m} \rangle $, $ k, m \in \N $. Somehow these groups seemed to be the natural next step upwards from the free group case as they only have one relation and can be realized as an amalgamated free product of copies of $ \Z $. 

The main result of the paper is an inclusion of the group $ C^* $-algebra of an amalgamated free product of Abelian groups into an algebra which is MF in the sense of Blackadar and Kirchberg, \cite{BK}. Since the MF property passes to $ C^* $-subalgebras this result gives the first examples of reduced amalgamated free products of $ C^* $-algebras with amalgamation over an infinite-dimensional $ C^* $-subalgebra which are MF, c.f., the recent work of Li and Shen \cite{LS}. See \cite{V} for more on reduced amalgamated free products and their relation to free probability.

Since almost none of the treated groups are amenable, we get the existence of non-invertible extensions of the corresponding reduced group $ C^* $-algebras by $ \mathbb{K} $ as easy corollaries.

\subsubsection*{Acknowledgements}
It is a pleasure to thank my advisors Steen Thorbj{\o}rnsen and Klaus Thomsen for many helpful discussions and suggestions during the preparation of this manuscript. A special thanks goes to Uffe Haagerup for various indispensable inputs. Part of this manuscript was written while visiting Nate Brown at the Pennsylvania State University and I wish to express my gratitude for the hospitality shown to me there.

\section{Preliminaries}
Most of the paper is concerned with the reduced group $ C^* $-algebras of discrete groups and we recall some standard facts.

If $ G $ is a discrete group we let $ C^*_r(G) $ denote the reduced group $ C^* $-algebra associated to $ G $, i.e., the unital $ C^* $-subalgebra of $ B(\ell^2(G)) $ generated by the left-regular representation. We let $ L(G) := C_r^*(G)'' $ be the group von Neumann algebra. The algebra $ L(G) $, and hence also $ C_r^*(G) $, is endowed with a faithful tracial state which can be realized as the vector state corresponding to any  $ 1_{\{g\}}  \in \ell^2(G) $, $ g \in G $, where $ 1_{\{g\}} $ is the characteristic function corresponding to the singleton $ \{g\} $.

If $ H $ is a subgroup of $ G $, then $ C^*_r(H)  $ is a unital subalgebra of $ C^*_r(G) $. 
If we have a family of discrete groups $ G_i $, $ i \in I $ with a common subgroup $ H $ we denote by $ \star_H G_i $ their amalgamated free product. If there is no amalgamation (i.e., if $ H = \{1\} $), we simply write $ \star G_i $ for the free product.

We will need the following elementary observation on reduced group $ C^* $-algebras.

\begin{lemma} \label{Induct}
Let $ G_i $, $ i \in \N $ be discrete groups such that $ G_i \subseteq G_{i+1} $ for all $ i $ and let $ G $ be the (direct) union of the $ G_i $. Then we have an isomorphism of $ C^* $-algebras
\begin{equation*}
\lim_k C^*_r(\cup_{i=1}^k G_i) \simeq C_r^*(G),
\end{equation*}
where the inductive limit is taking with respect to the natural inclusions $ C_r^*(\cup_{i=1}^k G_i) \to C^*_r(\cup_{i=1}^{k+1} G_i)$.
\end{lemma}

\begin{proof}
We have the standard inclusions of $ C^* $-algebras corresponding to the inclusion of subgroups which for each $ k $ make the following diagram commute
\begin{equation*}
\xymatrix{C_r^*(\cup_{i=1}^k G_i) \ar[d] \ar[r] & C_r^*(G) \\
C^*_r(\cup_{i=1}^{k+1} G_i) \ar[ur]} 
\end{equation*}

The induced $ * $-homomorphism $ \varphi: \lim_k C_r^*(\cup_{i=1}^k G_i) \to  C_r^*(G)  $ is readily seen to be an isomorphism. 
\end{proof}

It is a standard fact that if a $ C^* $-algebra $ A $ is MF but not quasidiagonal, then $ A $ has an extension by $ \mathbb{K} $ which is not invertible in the sense of Brown, Douglas and Fillmore. By a result of Rosenberg (see Theorem V.4.2.13 in \cite{B} for an elegant proof) $  C^*_r(G)$ is not quasidiagonal if $ G $ is a non-amenable group. Thus to prove the existence of a non-invertible extension of $ C^*_r(G) $ by $ \mathbb{K} $, i.e., that the extension semigroup, $ \Ext(C_r^*(G)) $, is not a group, one needs only establish the MF-property of the algebra and realize that the group is not amenable.

\section{The result}

Our main result will be a consequence of several lemmas which we prove below. The first lemma is a mere observation.

\begin{lemma} \label{GNSequi}
Let $ A, B $ be unital $ C^* $-algebras with a surjective $ * $-homomorphism $ \pi: A \to B $ and a state $ \varphi: B \to \C $. Let $ \tilde{\varphi} = \varphi \circ \pi $, then the GNS-representation corresponding to $ \tilde{\varphi} $ is unitarily equivalent to $ \pi_{\varphi} \circ \pi $ where $ \pi_{\varphi} $ is the GNS-representation corresponding to $ \varphi $.
\end{lemma}

\begin{proof}
It suffices to show that $ (\pi_{\varphi} \circ \pi, H_{\varphi}, 1_B) $ is a GNS-triple for $ \tilde{\varphi} $. Clearly, $ 1_B $ is a cyclic vector and
\begin{equation*}
\tilde{\varphi}(a) = \varphi(\pi(a)) = \langle \pi_{\varphi}(\pi(a)) 1_B, 1_B \rangle 
\end{equation*}
when $ a \in A $. This proves the claim.
\end{proof}

We briefly introduce a notion from harmonic analysis that will prove useful to us. 

For a unitary representation $ \sigma $ of a (discrete) group $ G $, we let $ h_\sigma $ denote the $ * $-homomorphism on the full group $ C^* $-algebra, $ C^*(G) $, that extends $ \sigma $.
\begin{defn}
If $ \sigma, \tau $ are unitary representations of the discrete group $ G $, we say that $ \sigma $ is \emph{weakly contained} in $ \tau $ and write $ \sigma \prec \tau $ if $ \ker h_\tau \subseteq \ker h_\sigma $.
\end{defn}

We refer to the books \cite{Di} and \cite{BHV} for the basics of this concept.

The following result which is interesting in its own right, is an important ingredient in our proof. The result is in fact an immediate consequence of the so-called 'continuity of induction' due to J.M.G. Fell applied to the trivial representation of the subgroup $ H $, see, e.g., \cite{BHV} Theorem F.3.5. We give a self-contained proof below.

\begin{prop} \label{uffe}
Let $ G $ be a discrete group with a normal subgroup $ H $. The canonical quotient map $ q : G \to G/H $ extends to a $*$-homomorphism $ \pi : C_r^*(G) \to C_r^*(G/H) $ if and only if $ H $ is amenable.
\end{prop}

\begin{proof}
Suppose $ \pi : C_r^*(G) \to C_r^*(G/H)  $ extends $ q $. We have a natural inclusion $ \iota: C_r^*(H) \to C_r^*(G) $ and since $ \pi \circ \iota(h) = 1 $ for all $ h \in H $ we get a $ * $-homomorphism $ \pi \circ \iota : C^*_r(H) \to \C  $ that sends all group elements to $1$. In other words the trivial representation of $ H $ is weakly contained in the left regular representation of $ H $ and consequently, $ H $ is amenable.

Conversely assume that $ H $ is amenable. We wish to find a $ * $-homomorphism $ \pi:C^*_r(G) \to C_r^*(G/H) $ that will make
\begin{equation*}
\xymatrix{C^*(G) \ar[d] \ar[r] & C_r^*(G/H) \\
C^*_r(G) \ar@{-->}[ur]} 
\end{equation*}
commute.

By Proposition 8.5 of \cite{P} we may find a net of unit vectors $ (f_i) \subseteq \ell^2(H) $ such that
\begin{equation*}
1 = \lim (f_i * \tilde{f_i})(h)
\end{equation*}
for all $ h \in H $. The function $ \tilde{f_i} $ is given by $ \tilde{f_i}(h) = \overline{f_i(h^{-1})} $, $ h \in H $ and $ * $ denotes convolution.
Define a net of unit vectors in $ \ell^2(G) $ by
\begin{equation*}
g_i(k) = \begin{cases} f_i(k) & \text{ if } k \in H \\
0 & \text{ if } k \notin H
\end{cases}
\end{equation*}
Then 
\begin{equation*}
1_H(k) = \lim (g_i * \tilde{g_i})(k)
\end{equation*}
for all $ k \in G $. 
We wish to define a state $ \psi $ on $ C^*_r(G) $ which is equal to $ 1_H $ on $ G $.
To see that this is possible let $ \sum_n \gamma_n k_n \in \C[G] $, then
\begin{align*}
\left \langle \left ( \sum_n \gamma_n k_n \right) g_i, g_i \right \rangle & = \sum_n \gamma_n \sum_{k \in G} g_i(k_n^{-1}k) \overline{g_i(k)} \\
&= \sum_n \gamma_n \sum_{k \in G} g_i(k) \overline{g_i(k_nk)}\\
&= \sum_n \gamma_n (g_i *\tilde{g_i})(k_n^{-1}) \\
& \to \sum_{n} \gamma_n 1_H(k_n).
\end{align*}
Since $ g_i $ is a unit vector in $ \ell^2(G) $ for each $ i $, it follows from the CBS-inequality that $ 1_H$ may be extended to a map on $ C^*_r(G) $. This extension is clearly a state.

We may consider the canonical map $ \mu : C^*(G) \to C^*_r(G) $. By composition we get a state $ \varphi = \psi \circ \mu $ on $ C^*(G) $ which is nothing but the state on $ C^*(G) $ which equals $ 1_H $ on $ G $. The GNS-representation of this state is unitarily equivalent to the map in the top row of the diagram above and so the previous lemma tells us that the map we are looking for is the GNS-representation of $ \psi $. 
\end{proof}

Note that if $ H=G $, Proposition \ref{uffe} reduces to a well-known characterization of amenability.

Like Proposition \ref{uffe} above, the following result is also due to Fell.

\begin{lemma} \label{Fell2}
Let $ \lambda $ denote the left regular representation of the discrete group $ G $. If $ \sigma $ is a unitary representation such that $ \sigma \prec \lambda $, then $ \sigma \otimes \pi \prec \lambda $ for any unitary representation $ \pi $ of $ G $.
\end{lemma}

\begin{proof}
By Fell's absorption principle, see, e.g., Theorem 2.5.5 of \cite{BO}, and Proposition F.3.2 of \cite{BHV}
\begin{equation*}
\sigma \otimes \pi \prec \lambda \otimes \pi \sim \lambda^{(\dim\pi)} \prec \lambda.
\end{equation*}
The equivalence, $ \sim $, being unitary equivalence.
\end{proof}

For a discrete Abelian group $ G $ we will let $ \hat{G} $ denote the unitary dual of $ G $. That is, $ \hat{G} $ is the set of unitary equivalence classes of irreducible unitary representations of $ G $. These are by Schur's Lemma all one-dimensional and so $ \hat{G} $ may be identified with the character space of $ C^*(G) $, i.e., $ C^*(G) \simeq C(\hat{G}) $. The unitary dual is a compact Hausdorff group in the topology of pointwise convergence on $ G $ (= the weak* topology from $ C^*(G) $) under pointwise multiplication.

\begin{lemma} \label{metrizable}
Let $ G $ be a discrete Abelian group. If $ G $ is countable then $ \hat{G} $ is metrizable. 
\end{lemma}

\begin{proof}
This is all very standard. If $ G $ is countable, $ C(\hat{G}) = C^*(G) $ is separable. Take a dense sequence $ f_n, $ $ n \in \N $ in $ C(\hat{G}) $ and define a metric $ d $ on $ \hat{G} $ by
\begin{equation*}
d(x,y) = \sum_n \frac{1}{2^{n}\|f_n \|} |f_n(x) - f_n(y) |.
\end{equation*}
One easily checks, by use of Urysohn's Lemma, that $ d $ is indeed a metric. To see that $ d $ induces the right topology on $ \hat{G} $ let $ x_\lambda \to x $ in $ \hat{G} $ and let $ \varepsilon > 0 $, then
\begin{equation*}
d(x_\lambda, x) \leq \sum_{n=1}^N \frac{1}{2^{n}\|f_n \|} |f_n(x_\lambda) - f_n(x) | + \tfrac{\varepsilon}{2}
\end{equation*}
for a suitable $ N $ independent of $ \lambda $ and so $ \lim_\lambda d(x_\lambda, x) = 0 $. The identity map on $ \hat{G} $ is then a continuous bijection from a compact space to a Hausdorff space, hence a homeomorphism.
 \end{proof}
 
 \begin{remark}
 The converse of Lemma \ref{metrizable} is of course true but we will not need it in the following.
 \end{remark}
 
\begin{theorem} \label{newmain}
Let $ G_i $, $ i \in I $ be a finite or countably infinite collection of countable discrete Abelian groups with a common subgroup $ H $. Let $ G $ denote the amalgamated free product $ \star_H G_i $ of the $ G_i $'s with amalgamation over $ H $. Then we have an injective $ * $-homomorphism
\begin{equation*}
C^*_r(G) \to C^*_r(\star (G_i/H)) \otimes L^\infty(\hat{H}),
\end{equation*}
where the algebra of equivalence classes of measurable essentially bounded functions is with respect to Haar measure on $ \hat{H} $. 
\end{theorem}

\begin{proof}
Since $ G/H = \star (G_i/H)$, Proposition \ref{uffe} gives a unital $ * $-homomorphism $ \rho :C^*_r(G) \to C^*_r(\star (G_i/H)) $. Consequently, $ \rho $ determines a unitary representation of $ G $ which is weakly contained in the left regular representation.

It is a classical result in harmonic analysis that the restriction map gives a homeomorphic isomorphism of compact Hausdorff groups
\begin{equation*}
\hat{G_i}/\{ \omega \in \hat{G_i} \mid \omega(h) =1, \ h \in H \} \to \hat{H}.
\end{equation*}
See, e.g., Theorem 4.39 in \cite{F}.

Now, since a surjective continuous map between compact metric spaces admits a Borel section (see for instance \cite{BR} for an elegant proof of this classical fact), we have, by Lemma \ref{metrizable} a Borel section for the quotient map
\begin{equation*}
 \hat{G_i} \to \hat{G_i}/\{ \omega \in \hat{G_i} \mid \omega(h) =1, \ h \in H \}.
\end{equation*}
Let $ e_i: \hat{H} \to \hat{G_i} $ be the Borel map obtained by composition of the maps just considered, i.e., $ e_i(\omega) $ is a (measurable) choice of an extension of $ \omega $ to a character on all of $ G_i $. 

Consider the group homomorphism $ \sigma_i  $ from $ G_i $ to the unitary group of $ L^\infty(\hat{H}) $ given by sending $ g \in G_i $ to the function
\begin{equation*}
\omega \mapsto (e_i (\omega))(g)
\end{equation*}
for $ \omega \in \hat{H} $.

Since $ \sigma_i(h) = \sigma_j(h) $ for all $ i, j $ and $ h \in H $, we get an induced unitary representation $ \sigma $ of $ G $ and by Lemma \ref{Fell2} we then get a $ * $-homomorphism $\psi : C^*_r(G) \to C^*_r(\star (G_i/H)) \otimes L^\infty(\hat{H})  $ by considering the tensor product of the unitary representation corresponding to $ \rho $ with $\sigma $.

\bigskip

We show that $ \psi $ is injective by considering the trace. Indeed, consider the tensor product of the standard trace $ \tau $ on $ C^*_r(\star (G_i/H))   $ and the normalized Haar measure $ m $ on $ L^\infty(\hat{H}) $. By composing this tensor product with $ \psi $ we get a trace on $ C^*_r(G) $. We claim that this trace is equal to the canonical faithful trace on $ C^*_r(G) $ which will ensure the injectivity of $ \psi $. To see that this is true, it suffices to show that  
\begin{equation*} \label{trace}
 \tau \otimes m(\psi(g)) = \begin{cases} 0 \quad \text{ if } g \neq 1 \\ 1 \quad \text{ if } g = 1 \end{cases} 
\end{equation*}
for $ g \in G \subseteq C^*_r(G) $.

Clearly, the equation above is satisfied for $ g=1 $. The equation is obviously also satisfied if $ \rho(g) \neq 1$. 

If $ \rho(g)=1 $ and $ g \neq 1 $, then $ g \in H \backslash \{1\} $ so by the Gelfand-Raikov Theorem (Theorem 3.34 of \cite{F}) there is $ \omega_0 \in \hat{H} $ such that $ \omega_0(g) \neq 1 $ and so by invariance of Haar measure
\begin{equation*}
\tau \otimes m(\psi(g))= m(\sigma(g)) = \int_{\hat{H}} \omega(g) \, dm(\omega) = \omega_0(g) \int_{\hat{H}} \omega(g) \, dm(\omega),
\end{equation*}
which implies that $ \tau \otimes m(\psi(g))=0 $.
\end{proof}

The representation $ \sigma $ in the proof of Theorem \ref{Hovedresultat} may seem like a somewhat mysterious object but in concrete cases it may have a nice description as the following example shows.

\begin{example}
Consider the torus knot group $ \Gamma_{k,m} = \langle a_1, a_2 \mid a_1^k a_2 ^{-m} \rangle $ from the introduction. This can be realized as the amalgamated free product $ \Z \star_\Z \Z $ where the subgroup embeds in the first factor by multiplication by $ k $ and in the second by multiplication by $ m $. Let for $ r \in \N $, $ \varphi_r : \mathbb{T}= \hat{\Z} \to \C $ be given by $ \varphi_r(e^{i t}) = e^{it/r} $, $ t \in [0,2\pi) $ and consider the unitary representation $ \zeta $ of $ \langle a_1, a_2 \mid a_1^ka_2^{-m} \rangle $ on $ L^2(\mathbb{T}) $ given by
\begin{equation*}
\zeta (a_1) = \varphi_k \quad \text{and} \quad \zeta(a_2) = \varphi_m
\end{equation*}
where the functions $ \varphi_r $ are identified with the multiplication operators they induce on $ L^2(\mathbb{T}) $.
This naturally occuring representation is exactly (a choice of) the representation $ \sigma $ in the proof above.
\end{example}

The following result is an immediate consequence of Theorem \ref{newmain} and \cite{HLSW}.

\begin{theorem} \label{MF eksempler}
Let $ G_i $, $ i \in I $ be a finite or countably infinite collection of countable discrete Abelian groups with a common subgroup $ H $. Then $ C^*_r(\star_H G_i) $ is MF.
\end{theorem}

\begin{proof}
Since the class of MF algebras is stable under inductive limits it suffices to show the claim when $ I $ is finite by Lemma \ref{Induct}.
 
Suppose $ I $ is finite. Then
\begin{equation*}
C^*_r(\star (G_i/H))=\star_\C C^*_r(G_i/H),
\end{equation*}
where the right-hand side is the reduced free product with respect to the standard traces on the group $ C^* $-algebras, furthermore this right-hand side is MF by Theorem 3.3.3 of \cite{HLSW}. From Theorem \ref{newmain} we get an inclusion
\begin{equation*}
C^*_r(\star_{H} G_i) \subseteq C^*_r(\star (G_i/H)) \otimes L^\infty(\hat{H}).
\end{equation*}
It follows from this and Proposition 3.3.6 of \cite{BK} that $ C^*_r(\star_{H} G_i) $ is MF since abelian $ C^* $-algebras are nuclear and MF and $ C^* $-subalgebras of MF algebras are MF.
\end{proof}

As mentioned in the introduction the above result is related to the theory of $ C^* $-extensions via the following lemma.

\begin{lemma} \label{amenabilitet}
Let $ I $ be a possibly infinite set with $ |I| \geq 2 $, $ G_i $, $ i \in I $ be a collection of countable discrete groups with a common normal subgroup $ H $ of index greater than or equal to $ 2 $ in each $ G_i $. Then $ \star_H G_i $ is amenable if and only if $ I $ has cardinality $ 2 $ and $ H $ is amenable with index $ 2 $ in each $ G_i $.
\end{lemma}

\begin{proof}
If $ I $ has cardinality $|I| = 2 $, $ H $ is amenable and has index $ 2 $ in both groups $ G_i $, we see that $ \star_H G_i $ is an extension of amenable groups whence $\star_H G_i  $ is amenable.

Obviously amenability of $ \star_H G_i  $ forces $ H $ to be amenable too.

If the cardinality of $ I $ is greater than or equal to $  3$ then $ \star (G_i/H) $ contains a subgroup isomorphic to one of the following groups
\begin{equation*}
 \Z_{p_1} \star \Z_{p_2} \star \Z_{p_3}, \  \Z \star \Z_{p_1}, \  \mathbb{F}_2,
 \end{equation*}
where each integer $ p_i \geq 2 $.

Each of these groups is non-amenable. The first one by Proposition 14.2 of \cite{P}, the second one has the non-amenable group (again by Proposition 14.2 \cite{P}) $ \Z_3 \star \Z_{p_1} $ as homomorphic image and hence cannot itself be amenable. Finally everyone knows that $ \mathbb{F}_2 $ is not amenable. From this it follows that $ \star (G_i/H) $ and hence $ \star_H G_i $ is not amenable when we are dealing with $ 3 $ or more groups.

If $ |I| = 2 $, and at least one of the indices is strictly greater than $ 2 $, $ \star (G_i/H) $ contains a subgroup isomomorphic to one of the following
\begin{equation*}
\mathbb{F}_2 , \  \Z_{p_1} \star \Z, \  \Z_{p_1} \star \Z_{p_2}, \  \Z_2 \star (\Z_2 \oplus \Z_2)=\langle a,b,c \mid a^2=b^2=c^2= cbc^{-1}b^{-1}=1 \rangle , 
\end{equation*}
where $ p_1 \geq 2 $ and $ p_2 \geq 3 $.

We have already noted that the first three groups are non-amenable. The last group is non-amenable because the subgroup generated by the elements $ abab $ and $ acac $ is isomomorphic to $ \mathbb{F}_2 $. We leave the tedious but straightforward argument to the reader. This completes the proof.
\end{proof}

\begin{cor}
Let $ I $ be a countable or finite set with $ |I| \geq 2 $, $ G_i $, $ i \in I $ a collection of countable, discrete Abelian groups with a common subgroup $ H $ which has index greater than or equal to $  2 $ in each $ G_i $. Then the BDF semigroup of extensions by $ \mathbb{K} $, $ \Ext(C_r^*(\star_H G_i)) $, is a group if and only if $2= |I|=|G_i/H| $ for both $ i \in I $.
\end{cor}

\begin{proof}
Combine Theorem \ref{MF eksempler} and Lemma \ref{amenabilitet} with the fact that $ \Ext(C_r^*(\star_H G_i)) $ is a group if $ C_r^*(\star_H G_i) $ is nuclear, see, e.g., \cite{Ar}.
\end{proof}

The proof of Theorem \ref{newmain} has some von Neumann algebra flavour to it. Indeed, in the proof we use a measurable section of a certain surjective map. In the example following the proof this corresponds to taking $ k $'th roots in $ \C $ which cannot possibly be done continuously.

Exploring this von Neumann aspect we get an isomorphism on the level of von Neumann algebras. A result which may be known to experts. The precise statement is as follows.

\begin{theorem} \label{Hovedresultat}
Let $ G_i $, $ i \in I $ be a finite or countably infinite collection of countable discrete Abelian groups with a common subgroup $ H $. Set $ \star_H G_i = G $. Then we have an isomorphism of von Neumann algebras
\begin{equation*}
L(G) \xrightarrow{\sim} L(\star (G_i/H)) \bar{\otimes} L^\infty(\hat{H}),
\end{equation*}
where the tensor product is the spatial tensor product of von Neumann algebras.
\end{theorem}

\begin{proof}
The injective $ * $-homomorphism from Theorem \ref{newmain} extends to a normal $ * $-homomorphism $ \psi:L(G) \to L(\star (G_i/H)) \bar{\otimes} L^\infty(\hat{H}) $ on the von Neumann algebra level. 

Injectivity of $ \psi $ follows in exactly the same way as it did in the proof of Theorem \ref{newmain} since the trace is also faithful on the von Neumann algebra level.

For surjectivity of $ \psi $ note first that the algebra $ 1 \otimes L^\infty(\hat{H}) $ is in the image of $ \psi $. Indeed, combining the classical theorems of Stone-Weierstrass and Gelfand-Raikov, we see that the algebra generated by $ \psi(H) $ is norm dense in $1 \otimes C(\hat{H}) $ which in turn is strongly dense in $1 \otimes L^\infty(\hat{H}) \subseteq  B(\ell^2(\star (G_i/H)) \otimes L^2(\hat{H})) $. On the other hand, for any $ b \in \star (G_i/H) \subseteq L(\star (G_i/H)) $ there is a unitary $ c \in L^\infty(\hat{H}) $ such that $ b \otimes c $ is in the image of $ \psi $. It follows that the image of $ \psi $ contains anything of the form $ a \otimes b $ where $ a \in L(\star (G_i/H)) $ and $ b \in L^\infty(\hat{H}) $ and thus $ \psi $ is surjective.
\end{proof}

\section{More Examples of MF Algebras}

The ideas from the last section can be used to give some new examples of reduced group $ C^* $-algebras with the MF property.

The simple strategy above is basically reducing the MF question of amalgamated free products of groups to the same question with no amalgamation. More precisely given a collection of discrete groups $ (G_i) $ with a common normal subgroup $ H $ the idea is to find a nuclear algebra with the MF property (i.e., an NF algebra) $ A $ such that
\begin{equation*}
C^*_r( \star_H G_i) \subseteq C^*_r(\star (G_i/H)) \otimes A,
\end{equation*}
and then from this deduce that the left-hand side is MF if $  C^*_r(\star (G_i/H)) $ is.

Executing the strategy, of course, requires a candidate for the algebra $ A $ and in the same breath a candidate for the map realizing the inclusion. This is where the real work lies. In the following we will discuss this line of attack in the case of a tower of groups.

\begin{prop} \label{tower}
Suppose $ H \subseteq G_1 \subseteq G_2 \subseteq \dots \subseteq G_n $ are discrete groups with $ H $ amenable and normal in each $ G_i $. Then we have an inclusion of $ C^* $-algebras
\begin{equation*}
C^*_r( \star_H G_i) \subseteq C^*_r(\star (G_i/H)) \otimes_{\min} C^*_r(G_n).
\end{equation*}
\end{prop}

\begin{proof}
The proof is very similar to that of Theorem \ref{newmain} so we will be brief. Note that the free product of the group inclusions gives a map $ \star_H G_i \to G_n \subseteq C^*_r(G_n) $. This representation tensored with the representation $ \star_H G_i \to (\star_H G_i)/H = \star (G_i/H) \subseteq C^*_r(\star (G_i/H)) $ is weakly contained in the left regular representation (Proposition \ref{uffe} and Lemma \ref{Fell2}) and so gives a map between the $ C^* $-algebras in question. The injectivity of this map is established by considering the trace once again.
\end{proof}

To use this observation to get new examples of MF algebras, we need to know that each $ C^*_r(G_i/H) $ is an ASH algebra so that we can invoke \cite{HLSW} and we need to know that $ C^*_r(G_n) $ is NF. The last condition is related to the well known (and somewhat notorious) conjecture of Rosenberg stating that $ C^*_r(G) $ is quasidiagonal for any countable discrete amenable group $ G $. Indeed, if $ C^*_r(G_n) $ is NF it must be quasidiagonal by Voiculescu's characterization of quasidiagonality and the Choi-Effros Lifting Theorem and $ G_n $ must be amenable.

In other words to get concrete examples, $ G_n $ must be an amenable group for which Rosenberg's conjecture holds. Unfortunately Rosenberg's conjecture has not, to the knowledge of the author, been established in very many cases. Obvious classes for which the conjecture holds are the abelian and the finite groups. Another concrete example based on what we have done so far is the following.

\begin{example}
Let $ S_\infty $ be the group of permutations on $ \N $ with finite support. It has a normal subgroup $ A_\infty $ consisting of the finitely supported even permutations. By Lemma \ref{amenabilitet} the group $ S_\infty \star_{A_\infty} S_\infty $ is amenable and clearly $ C^*(S_\infty) $ is NF (it is in fact AF), so by Proposition \ref{tower} $ S_\infty \star_{A_\infty} S_\infty $ has a group $ C^* $-algebra which is MF and the discussion above tells us that $ S_\infty \star_{A_\infty} S_\infty $ satisfies Rosenberg's conjecture. Note that we do not have to invoke the results of \cite{HLSW} since the quotient $ S_\infty \star_{A_\infty} S_\infty/A_\infty $ is the infinite dihedral group which is amenable and residually finite and hence has a group $ C^* $-algebra which is MF. See details below.
\end{example}

We will call a group MF if its reduced group $ C^* $-algebra is MF. Using standard terminology from group theory a group $ G $ is then residually MF if there are MF groups $ G_n $, $ n \in \N $, such that $ G \subseteq \prod_n G_n $.

\begin{prop} \label{embedding}
Suppose $ G $ is a discrete group and that there is a sequence of discrete groups $ (G_n)_{n \in \N} $ and surjective homomorphisms $ \varphi_n : G \to G_n $ satisfying that $ \ker \varphi_n $ is amenable for each $ n $ and that
\begin{equation*}
\bigcap_{n \in I} \ker \varphi_n = \{1\},
\end{equation*}
whenever $ I  $ is an infinite subset of $ \N $.

It follows that we have an embedding 
\begin{equation*}
 \pi: C^*_r(G) \to \prod_{n=1}^\infty C^*_r(G_{n}) \Big /    \sum_{n=1}^\infty C^*_r(G_{n}). 
\end{equation*}
\end{prop} 

\begin{proof}
Lemma \ref{uffe} ensures the existence of the $ * $-homomorphism $ \pi $. As usual the injectivity follows by considering the trace.

Fix a character $ \omega $ on $ \ell^\infty(\N) $ such that $ \omega $ is not evaluation at any point $ n \in \N $. Note that $ \ker(\omega) $ contains any sequence that is eventually $ 0 $ and so by continuity all sequences converging to $ 0 $. In fact if $ x \in \ell^\infty(\N) $ is a sequence of $ 0 $'s and $ 1 $'s, $ \omega (x) = \omega (x)^2 $, so $  \omega (x) $ is either $ 1 $ or $ 0 $. In particular if $x= 1_{\{n\}} $ and $ \omega(x)= 1 $ then by Urysohn's Lemma there is an $ 1 \geq f \geq 0 $ in $ \ell^\infty(\N) $ with $  f(n) =0$ and $ \omega(f) > 0 $ since $ \omega $ is not evaluation at $ n $. Then 
\begin{equation*}
 \| f + 1_{\{ n \}} \| \geq \omega(f + 1_{\{ n \}})> 1
\end{equation*}
but the norm equals $ 1 $. It follows that $ \omega (1_{\{ n \}})=0 $ for all $ n \in \N $.

Let $ \tau_n $ denote the standard trace on $ C_r^*(G_{n}) $. The equation
\begin{equation*}
\tau((x_n)) =\omega( (\tau_n(x_n)))
\end{equation*}
defines a trace on $ \prod_{n=1}^\infty C^*_r(G_{n}) $ which drops to
\begin{equation*}
\prod_{n=1}^\infty C^*_r(G_{n}) \Big /    \sum_{n=1}^\infty C^*_r(G_{n})
\end{equation*}
since the ideal is killed by $ \tau $. 

We consider the trace $ \tau \circ \pi$ on $ C^*_r(G) $. Let $ g \in G $. Then
\begin{equation} \label{01}
\tau \circ \pi (g) = \tau ( (\varphi_n(g))) = \omega (\tau_n (\varphi_n(g))),
\end{equation}
where $ \varphi_n $ is the quotient map $ G \to G_n $.

Since $ (\tau_n (\varphi_n(g))) $ is a sequence of $ 0 $'s and $ 1 $'s \eqref{01} equals $ 0 $ or $ 1 $ by the reasoning above. If \eqref{01} is $ 1 $ then $ \tau_n (\varphi_n(g)) =1 $ for infinitely many $ n $ which in turn implies that $ \varphi_n(g) =1 $ for infinitely many $ n $ and so $ g =1 $ by assumption. In short
\begin{equation*}
\tau \circ \pi (g) = \begin{cases} 1 & \text{ if } g =1 \\
0 & \text{ if } g \neq 1 \end{cases}.
\end{equation*}
It follows that $ \tau \circ \pi $ is exactly the usual faithful trace on $ C^*_r(G) $ and so $ \pi $ is injective.
\end{proof}

\begin{cor}
Suppose $ G $ is an amenable residually MF group. It follows that $ G $ is MF and hence satisfies Rosenberg's conjecture.
\end{cor}

\begin{proof}
Let $ G_n $, $ n \in \N $ be MF groups such that $ G \subseteq \prod_n G_n $. We may assume that each $ G_n $ is amenable and hence $ \prod_{n}^k G_n $ is MF. Now, the maps $ G \subseteq \prod_n G_n \to \prod_{n}^k G_n $ satisfies the hypothesis of Proposition \ref{embedding}. An application of Corollary 3.4.3 of \cite{BK} shows that $ G $ is MF.
\end{proof}

For instance all amenable residually finite groups satisfy Rosenberg's conjecture as can also easily be observed since their group $ C^* $-algebras are even residually finite dimensional.

In the setting of Proposition \ref{tower} with $ G_n $ amenable and, say residually finite, we can conclude that the group $ \star_H G_i $ is MF. We end the exposition by giving an example of an MF group which is a free product of non-abelian groups with infinite amalgamation.

\begin{example}
Consider the set
\begin{equation*}
G= \left\{ \begin{pmatrix} 1 & x & y \\ 0 & I & z \\ 0 & 0 & 1 \end{pmatrix} \Bigg| x,z \in \Z^n, \ y \in \Z \right\}
\end{equation*}
for $ n \in \N $. This is a (non-Abelian) group under matrix multiplication called the discrete Heisenberg group. It has a normal subgroup isomorphic to $ \Z^{n+1} $ consisting of the matrices with $ x=0 $ in the above presentation. Taking the quotient with this, we get $ \Z^{n} $, so in particular $ G $ is amenable. Let for $ k \in \N $, $ G_k $ be the normal subgroup of $ G $ given by
\begin{equation*}
G_k= \left\{ \begin{pmatrix} 1 & kx & ky \\ 0 & I & kz \\ 0 & 0 & 1 \end{pmatrix} \Bigg| x,z \in \Z^n, \ y \in \Z \right\}.
\end{equation*}
This has finite index in $ G $ and it follows that $ G $ is residually finite and hence satisfies Rosenberg's conjecture. Furthermore $ G \star_{G_k} G $ is an MF group for each $ k $, note here that the subgroup we amalgamate over is infinite and non-Abelian.

The infinite dihedral group provides another example of a residually finite amenable group and the Baumslag-Solitar groups $B(1,m) := \langle a,b \mid a^{-1} b a = b^m \rangle $, $ m \in \N $ also have these properties and so we could form suitable amalgamated free products of these to get more MF groups.

\end{example}

\end{document}